\documentclass[12pt, a4paper]{article}
\usepackage[T1]{fontenc}
\usepackage{amsmath,amssymb,amsfonts,amscd}
\usepackage{enumerate}
\usepackage{amsfonts,bbm}

\usepackage{amsthm}

\usepackage{mathrsfs}
\usepackage{amsfonts}
\usepackage[latin1]{inputenc}
\usepackage[T1]{fontenc}
\usepackage[english]{babel}
\usepackage{amssymb}

\usepackage{latexsym}
\usepackage{amsfonts}
\usepackage{syntonly}
\usepackage{tracefnt}
\usepackage{amsmath}
\usepackage{a4}

\newtheorem{lemm}{Lemma}[section]
\newtheorem {theo}[lemm]{Theorem}
\newtheorem{prop}[lemm]{Proposition}

 \newcommand{\R}{\mathbb R}
 
\newcommand{\C}{\mathbb C}

\begin{document}
\title{\bf Mean value property associated with the Dunkl Laplacian}
\author{Kods Hassine\\
{\small\em  Department of Mathematics, Faculty of Sciences, University of Monastir}\\  {\small\em 5019 Monastir, Tunisia }\\
{\small\em E-mail: hassinekods@gmail.com}
}
\date{}
\maketitle

\abstract{Let $\Delta_k$ be the Dunkl Laplacian on $\R^d$. The main goal of this paper is to characterize $\Delta_k$-harmonic functions by means of a mean value property.}

{\small 


}

\section{Introduction}
\label{intro}


Let $R$ be a root system of $\R^d$, $d\geq1$,  $k:R\rightarrow \R_+$ be a multiplicity function and   $W$ be the  group generated by the  reflections
$\sigma_\alpha$, $\alpha\in R$.
The Dunkl Laplacian is defined in \cite{dunkl2} for every function $f\in C^2(\R^d)$ by
$$\Delta_kf(x)=\Delta f(x)+2\sum_{\alpha\in R_+}k(\alpha)\left(\frac{<\nabla f(x),\alpha>}{<\alpha,x>}-\frac{|\alpha|^2}2\frac{f(x)-f(\sigma_\alpha(x))}{<\alpha,x>^2}\right),$$
where  $\Delta$ and $\nabla$ denote respectively the usual  Laplacian and gradient on~$\R^d$ and $R_+$ is a positive subsystem of $R$. Clearly, if $k$ is the
identically vanishing function, then $\Delta_k$
is reduced to~$\Delta$.
\par It is well known that a locally bounded function $f$ on an open subset $D$ of~$\R^d$,  is  $\Delta$-harmonic (i.e.,  $f\in C^2(D)$ and $\Delta f=0$ on $D$) if and only if
$$f(x)=\frac{1}{\sigma_{x,r(S(x,r))}}\int_{S(x,r)}f(y)d\sigma_{x,r}(y),$$
 for every $x\in D$ and every $r>0$ such that the closed ball $\overline{B}(x,r)$ of center $x$ and radius $r$ is contained in $D$. Here    $\sigma_{x,r}$ is the  surface area measure on the  sphere $S(x,r)$ with center $x$ and radius $r$.
 \par
H. Mejjaolli and K. Trimèche showed  in \cite{mt1} that every infinitely
 differentiable function $f$  on $\R^d$ is $\Delta_k$-harmonic on $\R^d$ if and only if for all $x\in \R^d$ and $r>0$,
\begin{equation}\label{pmoyenne}
f(x)=\frac{1}{d_k}\int_{S(0,1)}\tau_xf(ry)\left(\prod_{\alpha\in R_+}|\langle y,\alpha\rangle|^{2k(\alpha)}\right)d\sigma_{0,1}(y),
\end{equation}
where $d_k$ is a normalized constant and  $\tau_x$ is the Dunkl translation.
 The main goal of this paper is to  investigate a   mean value property  which characterizes the $\Delta_k$-harmonicity of locally bounded  functions on an open subset of $\R^d$.
\par  Let $D\subset \R^d$ be an open set which is $W$-invariant. We shall say that a function  $f: D\rightarrow \R$  satisfies the mean value property on $D$ if  for every $x\in D$ and $r>0$ such that $\overline B(x,r)\subset D$,
$$f(x)=\int_{\R^d}f(y)d\sigma_{x,r}^k(y),$$
 where $\sigma_{x,r}^k$  (see~\cite{rosler3})  is the  unique probability measure on $\R^d$ such that the right hand side of  (\ref{pmoyenne}) coincides with
$$\int_{\R^d}f(y)d\sigma_{x,r}^k(y).$$
\par We shall prove that every locally bounded function $f$ on $D$ is $\Delta_k$-harmonic if and only if it satisfies the mean value property on $D$.
To that end, we  prove first the equivalence  for infinitely differentiable functions on $D$. Next, we
show   that for  a locally bounded function $f$ on  $D$, if $f$ satisfies the mean value property then
$f$ is infinitely differentiable on $D$. Thus, $f$ is $\Delta_k$-harmonic provided it satisfies the mean value property on $D$. To prove the converse, we need only show that if $f$ is $\Delta_k$-harmonic then it is infinitely differentiable on~$D$. This will be proved once we have shown that the operator $\Delta_k$ is hypoelliptic. Thus, by means of convergence property of $\Delta_k$-harmonic functions,  we prove that the operator~$\Delta_k$ is hypoelliptic on~$D$.
\par Note that the condition that $D$ is $W$-invariant is nearly optimal.
In fact, in the case where $d=1$, for every open set $D\subset \R$ which is not $W$-invariant,  we can always  construct a $\Delta_k$-harmonic function function  $f$ on $D$ which  does not
satisfy the mean value property on~$D$.
\section{Preliminaries and some lemmas}
\label{sec:1}
Let $S(\R^d)$  be the Schwartz space and  $C_0(\R^d)$ be  the set of all continuous functions on $\R^d$ vanishing at infinity.
For every open set $U\subset \R^d$,  $C(U)$ and $C_c(U)$ will denote respectively the set of all continuous functions on $U$ and  the set of all continuous functions  with compact support on $U$. The set of all  bounded functions in $C(U)$  will be denoted by $C_b(U)$.
For every  $\alpha\in \mathbb R^d\backslash\{0\}$, let $H_\alpha$ be the hyperplane of~$\R^d$ orthogonal to $\alpha$ and let $\sigma_\alpha$ be the reflection in $H_\alpha$, i.e.,
$$\sigma_\alpha(x):=x-2\frac{<\alpha,x>}{|\alpha|^2}\alpha,$$
where $\langle x,y\rangle=\sum_{i=1}^d x_iy_i$  and $|x|:=\sqrt{\langle x,x\rangle}.$
A finite subset $R$ of $\mathbb R^d\setminus \{0\}$ is called a \emph{root system} if $R\cap\mathbb R \alpha=\{\pm\alpha\}$ and $\sigma_\alpha(R)=R$ for all $\alpha\in R$.
For a given root system $R$, we denote by $W$  the finite group generated by all  refections $\sigma_\alpha,\, \alpha\in R$.
A function $k: R\rightarrow \mathbb R_+$ is called a \emph{multiplicity function} if it satisfies  $k(w\alpha) = k(\alpha)$,   for every $w\in W$ and every $\alpha\in R$.
\par Throughout this paper we fix  a root system $R$,  a  multiplicity function $k$ and  a \emph{$W$-invariant open} subset $D$ of~$\R^d$, that is,  $w(D)\subset D$ for all $w\in W$.
Let $w_k$ be the \emph{weight function} on $\R^d$  defined by,
$$w_k(x):=\prod_{\alpha\in R_+}|\langle x,\alpha\rangle|^{2k(\alpha)},$$
where  $R_+:=\{\alpha\in R : \langle\alpha,\beta\rangle >0\}$ for  some
$\beta\in \R^d\setminus\cup_{\alpha\in R}H_\alpha$.
Note that $w_k$ is homogeneous of degree $2\gamma$, with $\gamma:=\sum_{\alpha\in R_+}k(\alpha).$
From now on, we assume that $$\lambda:=\gamma+\frac d2-1>0.$$
   {The  Dunkl Laplacian} associated with the root system $R$ and the multiplicity function $k$ is the operator
$$\Delta_k:=\sum_{i=1}^dT_{i}^2,$$
where for every $1\leq i\leq d$ and $f\in C^1(D)$,
$$T_i f(x):=\partial_i f(x)+\sum_{\alpha\in R_+}k(\alpha)\alpha_i\frac{f(x)-f(\sigma_\alpha( x))}{<\alpha,x>},\quad x\in D.$$
\par By \cite{dunkl1}, there exists a unique linear isomorphism $V_k$  from the space of homogenous polynomials of degree $n$
on $\mathbb R^d$ into it self such that  $V_k1=1$ and
$T_i V_k=V_k\partial_i$. Later, it was shown in \cite{trim} that   the intertwining operator $V_k$  has an  homeomorphism  extension to $C^\infty(\R^d)$.
 The positivity of $V_k$ (see \cite{rosler1}) yields the existence of a family of probability measures $(\mu_x^k)_x$ such that for every $x\in\R^d$ and every $f\in C^\infty(\R^d)$,
$$V_kf(x)=\int_{\R^d}f(y)d\mu_x^k(y).$$ The support
of $\mu_x^k$ is contained  in the convex hull $C(x)$ of the orbit of $x$ under the reflection group $W$,
$$C(x):=co\{wx,\quad w\in W\}.$$
\emph{The Dunkl kernel} associated with $R$ and $k$ is defined on $\R^d\times \R^d$ by
$$E_k(x,y):=\int_{\mathbb R^d}e^{\langle y,\xi \rangle}d\mu_x^k(\xi).$$
It is well known that $E_k$ is positive, symmetric  and admits a unique holomorphic extension to $\mathbb C^d\times \mathbb C^d$ satisfying
$E_k(\xi z,\omega)=E_k(z,\xi\omega )$ for every $z,\omega\in\mathbb C^d$ and every
$\xi\in \mathbb C$.
The corresponding  \emph{Dunkl transform}  is then given for every bounded measure $\mu$ on $\R^d$ by
$$\mathcal F_D(\mu)(x):=c_k\int_{\mathbb R^d}E_k(-i\xi,x)d\mu(\xi),\quad x\in \R^d,$$
where $$c_k:=\left(\int_{\mathbb R^d}e^{-\frac{|y|^2}{2}}w_k(y) dy\right)^{-1}.$$
If $\mu=fw_kdx$ where   $f\in S(\R^d)$ and $dx$ is the Lebesgue measure on $\R^d$, then  we shall  write $\mathcal F_D(f)$ instead of $\mathcal F_D(\mu)$.
Note that $\mathcal F_D$ is  injective on the space of all bounded Borel measures $\mathcal M_b(\mathbb R^d)$ on $\R^d$ (see~\cite{rv}) and is a topological isomorphism from $S(\mathbb R^d)$ into it self (see~\cite{dejeu}).
 For each $x\in \R^d$, \emph{the Dunkl translation} $\tau_x$ is defined  for every   $f\in S(\mathbb R^d)$  by
 $$\tau_xf={\mathcal F}_D^{-1}(E_k(ix,\cdot)\mathcal F_Df),$$
 where  ${\mathcal F}_D^{-1}$ denotes the inverse of $\mathcal F_D$ on $S(\R^d)$.
 In \cite{trim2}, this translation was extended to $C^\infty(\R^d)$ by
 $$\tau_xf(y)=\int_{\R^d}\int_{\R^d}V_k^{-1}f(z+\eta)d\mu_x^k(z)d\mu_y^k(\eta),$$ where $V_k^{-1}$ is the inverse of $V_k$ on $C^\infty(\R^d)$.
 It was shown that, for every $f~\in~C^\infty(\R^d)$, the function  $u:(x,y)\mapsto \tau_xf(y)$ is symmetric,  infinitely differentiable on $\R^d\times \R^d$ and
 for every $x,y\in\R^d$,
  \begin{equation}\label{txi}
    (T_i)_x u(x,y)=(T_i)_y u(x,y).
\end{equation}
  Moreover, $\tau_xf(0)=f(0)$, $T_i\tau_xf=\tau_xT_{i}f$ and $\tau_xE_k(z,\cdot)(y)=E_k(x,z)E_k(y,z)$ for every  $z\in \C^d$.
 Further, if the support of $f$ (noted $\mbox{ supp\,}f$) is in $B(0,r)$ for some $r>0$,  then $\mbox{ supp }\,\tau_xf\subset B(0,r+|x|)$.
\par According to \cite{rosler3}, for each $x\in\R^d$ and $r>0$, there exists a unique  probability measure $\sigma_{x,r}^k$ on $\R^d$ which is supported by
$\bigcup_{w\in W} \overline B(wx,r)\setminus B(0,||x|-r|)$ such that for every $f\in C^\infty(\R^d)$,
\begin{equation}\label{hatem}
\frac{1}{d_k}\int_{S(0,1)}\tau_xf(ry)w_k(y)d\sigma_{0,1}(y)=\int_{\mathbb R^d}f(y)d\sigma^k_{x,r}(y).
\end{equation}
where,
$$d_k:=\int_{S(0,1)}w_k(y)d\sigma_{0,1}(y)=\frac{1}{c_k2^{\lambda}\Gamma(\lambda+1)}.$$
\begin{lemm}\label{lemmeelementaire}
Let  $\varphi\in S(\R^d)$ be radial. Then for every Borel set $A\subset \mathbb R^d$ and every
$x\in\mathbb R^d$,
 \begin{equation}\label{eq}\int_A
\tau_{-x}\varphi(y)
w_k(y)dy=d_k\int_0^\infty{\varphi}(t)t^{2\lambda+1}\left(\int_Ad\sigma^k_{x,t}(y)\right)dt.\end{equation}
\end{lemm}
\begin{proof}
Let $x\in\mathbb R^d$ and denote by $\mu(A)$ and $\nu(A)$ the left hand side and the right hand side respectively of (\ref{eq}).
Clearly, both  $\mu$ and $\nu$ are   bounded measures on $\R^d$. For every $y\in \R^d$,
\begin{eqnarray*}
{\mathcal F}_D(\mu)(y)&=&{\mathcal F}_D(\tau_{-x}\varphi)(y)\\
&=&{\mathcal F}_D\varphi(y)E_k(-iy,x)\\
&=&c_k E_k(-iy,x)
\int_{\mathbb R^d}\varphi
(z)E_k(-iy,z)w_k(z)dz\\
&=&c_k\int_{\R^d}\tau_{x}E_k(-iy,\cdot)(z)\varphi(z)w_k(z)dz
\end{eqnarray*}
Using  spherical coordinates and (\ref{hatem}), we deduce that,
\begin{eqnarray*}
{\mathcal F}_D(\mu)(y)&=& c_kd_k\int_0^\infty
t^{2\lambda+1}{\varphi}(t)\int_{\R^d}E_k(\xi,-iy)d\sigma_{x,t}^k(\xi)dt\\
&=&{\mathcal F}_D(\nu)(y).
\end{eqnarray*}
Finally,  we use the injectivity of $\mathcal F_D$ on   $\mathcal M_b(\mathbb R^d)$ to
conclude.\end{proof}
\vskip 5mm

 Let $\varphi\in S(\R^d)$ be a  radial function   with  support in $\overline B(0,r)$, $r>0$. We claim that for every $x\in \R^d$,
 \begin{equation}\label{cor}
 \mbox{ supp\,}\tau_x\varphi\subset \bigcup_{w\in W} \overline B(wx,r).
 \end{equation}
Indeed, let $A$ be a Borel subset of $\R^d\setminus\cup_{w\in W} \overline B(wx,r)$. Then by (\ref{eq}),
$$ \int_A\tau_{-x}\varphi(y)
w_k(y)dy=d_k\int_0^r{\varphi}(t)t^{2\lambda+1}\left(\int_Ad\sigma^k_{x,t}(y)\right)dt.$$
Since for every $0<t<r$, $\mbox{{ supp\,} }\sigma_{x,t}^k\subset\cup_{w\in W} \overline B(wx,r)$ we deduce that,
$$\int_A\tau_{-x}\varphi(y)
w_k(y)dy=0.$$
This proves the claim.

 In the sequel we shall write $$M_{x,r}(f)=\int_{\R^d}f(y)d\sigma_{x,r}^k(y),$$ whenever the integral makes sense.
A Borel function $f:D\rightarrow \R$ is said to satisfy  \emph{the mean value property} on $D$ if $M_{x,r}(f)=f(x)$ for every $x\in\R^d$ and $r>0$ such that $\overline B(x,r)\subset D$.

\begin{lemm}\label{moycinfty}
Let $f$ be a locally bounded function on $D$. If $f$ satisfies the mean value property on D, then $f\in C^\infty(D)$.
\end{lemm}
\begin{proof} Without loss of generality we suppose that $f$ is bounded on $D$.
 Let $\phi$ be the function defined for every $t\in \R$ by
$\phi(t):=ce^{-\frac 1t}{\chi}_{]0,\infty[}(t),$
 where ${\chi}_{]0,\infty[}$ is the indicator function of $]0,\infty[$ and the constant $c$ is chosen so that
$$cd_k\int_0^1\phi(1-t^2)t^{2\lambda+1}dt=1.$$
For every $n\geq 1$ we define the function $\phi_n$ by,
\begin{equation}\label{phin}
\phi_n(x)= n^{2\lambda+2}\phi(1-n^2|x|^2),\quad x\in\R^d.
\end{equation}
Obviously $\phi_n$ is infinitely differentiable on $\R^d$ with support  in $ \overline B(0,\frac 1n)$. Thus, by (\ref{cor}), for every $x\in\R^d$,
$$\mbox{ supp\,} \tau_{x}\phi_n\subset \cup_{w\in W}\overline B(wx,\frac 1n).$$
Let $D_n:=\{x\in D:\;  \overline B(x,\frac 1n)\subset D\}$ and let
  $$f_n(x):=\int_Df(y)\tau_{-x}\phi_n(y)w_k(y)dy,\quad x\in \R^d.$$   Then $f_n\in C^\infty(D_n)$. On the other hand, it follows from (\ref{eq}) that for every $x\in D_n$,
\begin{eqnarray*}
f_n(x)&=&d_k\int_0^{\frac 1n}{\phi_n}(t)t^{2\lambda+1}M_{x,t}(f)dt=f(x).
\end{eqnarray*}
 Hence   $f\in C^\infty(D_n)$ and consequently  $f\in C^\infty(D)$ as desired.\end{proof}
\vskip 5mm
\section{Main result}
\label{sec:2}
$D$ will always denotes a $W$-invariant open subset of $\R^d$. Our main result is the following:

\begin{theo}\label{promoy}
 Let   $f$ be a locally bounded function on D. The following statements are equivalent:
 \begin{description}
 \item[\it{(a)}] $f\in C^2(D)\mbox{ and }\Delta_k f=0 \mbox{ on }D$.
 \item[\it{(b)}]  $M_{x,r}(f)=f(x)$ for every $x\in D$ and $r>0$ such that $\overline B(x,r)\subset D$ .
\end{description}
\end{theo}
The following proposition shows the equivalence between {\it{(a)} } and {\it{(b)}} whenever $f$ is infinitely  differentiable on $D$. First, let us recall the Green formula associated with the Dunkl Laplacian (see \cite{mt1}): For every $f\in C^2(\overline B(0,t))$, $t>0$,
\begin{equation}\label{green}
\int_{B(0,t)} \Delta_kf(y)w_k(y)dy=\int_{S(0,t)}\frac{\partial}{\partial n}f(y)w_k(y)d\sigma_{0,t}(y),
\end{equation}
where $\frac{\partial}{\partial n}$ is the partial derivation operator in the direction of the exterior unit normal.
 \begin{prop}\label{propmoycoro}
Assume that $f$ is infinitely differentiable on $D$. Then $f$ is $\Delta_k$-harmonic   on $D$ if and only if $f$ satisfies the mean value property on $D$.
 \end{prop}
 \begin{proof}
Let $x\in D$ and  $r>0$ such that  $ \overline B(x,r)\subset D$. We claim that
  $t~\mapsto~M_{x,t}(f) $ is derivable on $]0,r[$ and for every $t\in ]0,r[$,
\begin{equation}\label{ddtmd}
\frac
{d}{dt}M_{x,t}(f)=\frac{1}{t^{2\lambda+1}}\int_0^ts^{2\lambda+1}M_{x,s}(\Delta_kf)ds.
\end{equation}
Indeed, since for every $s\in]0,r[$,  the support
 of $\sigma_{x,s}^k$ is contained in $\cup_{w\in W}B(wx,r),$ it suffices to  prove (\ref{ddtmd}) replacing $f$ by a function  $h\in  C^\infty(\R^d)$ such that
 $$h=f \mbox{ on } \cup_{w\in W}B(wx,r).$$
It is easily seen  from (\ref{hatem}) that for every $t\in]0,r[$,
\begin{eqnarray*}
\frac{d}{dt}M_{x,t}(h)&=&\frac{1}{d_k}\int_{S(0,1)}\langle\nabla(\tau_xh)(ty),y\rangle w_k(y)d\sigma_{0,1}(y)\\
&=&\frac{1}{d_kt^{2\lambda+1}}\int_{S(0,t)}\langle\nabla(\tau_xh)(u),\frac
ut\rangle w_k(u)d\sigma_{0,t}(u)\\
&=&\frac{1}{d_kt^{2\lambda+1}}\int_{S(0,t)}\frac{\partial}{\partial
n}(\tau_xh)(u)w_k(u)d\sigma_{0,t}(u).
\end{eqnarray*}
Therefore, by  the Green formula (\ref{green})  and    the fact that $\Delta_k\tau_x=\tau_x \Delta_k$,
\begin{eqnarray*}\frac{d}{dt}M_{x,t}(h)&=&\frac{1}{d_kt^{2\lambda+1}}\int_{B(0,t)}\tau_x(\Delta_k h)(u)w_k(u)du.
\end{eqnarray*}
Hence, using  spherical coordinates we deduce that,
$$\frac{d}{dt}M_{x,t}(h)=\frac{1}{t^{2\lambda+1}}\int_0^ts^{2\lambda+1}M_{x,s}(\Delta_kh)ds.$$
Thus the claim is proved. Now assume that $\Delta_kf=0$ on $D$.  Then for all $t\in]0,r[$, $\frac{d}{dt}M_{x,t}(f)=0$, by  (\ref{ddtmd}).  This yields that
$M_{x,t}(f)= \lim_{s\rightarrow 0}M_{x,s}(f)$. On the other hand, it is known from  \cite{rosler3} that  the map $(x,s)\mapsto \sigma_{x,s}^k$ is continuous with respect to the weak topology on $\mathcal M_b(\R^d)$. Thus,
\begin{equation}\label{limmxr}
\lim_{s\rightarrow 0}M_{x,s}(f)=f(x).
\end{equation}
 Whence $M_{x,t}(f)=f(x)$ which yields the necessity . Conversely, assume that~$f$ satisfies the mean value property on $D$. Then, using (\ref{ddtmd}) we deduce
 that $M_{x,t}(\Delta_kf)=0$ for all $t\in]0,r[$. Letting $t$ tend to 0 we obtain  that
  $\Delta_kf(x)=0$.\end{proof}
\vskip 5mm
We then conclude, in virtue of  Lemma \ref{moycinfty}, that  every  locally bounded function $f$ on $D$ which satisfies the
 mean value property on $D$  is  necessarily $\Delta_k$-harmonic on $D$. The converse statement will be proved in the remainder of this section.
\begin{lemm}\label{theconver}
Let $(h_n)_{n\geq 1}\subset C^\infty(D)$ be a locally uniformly bounded sequence of  $\Delta_k$-harmonic functions on $D$ with pointwise limit $h$. Then $h~\in~ C^\infty(D)$ and $\Delta_kh=0$ on $D$.
\end{lemm}
\begin{proof} Let $x\in D$ and let $r>0$ such that $\overline B(x,r)\subset D$.
Since  for every $n\geq 1$  the function $h_n$ is $\Delta_k$-harmonic on $D$, it follows  from  Proposition~\ref{propmoycoro}  that,
$$h_n(x)=\int_{\R^d}h_n(y)d\sigma_{x,r}^k(y).$$
Applying the dominated convergence theorem, we get
$h(x)=M_{x,r}(h)$. Whence~$h$ satisfies the mean value property on $D$ which finishes the proof, by Lemma~\ref{moycinfty} and Proposition~\ref{propmoycoro}.\end{proof}
\vskip 5mm
 \par  Let $g_k$ be  \emph{the fundamental solution of the Dunkl Laplacian}.
 That is, for every $\varphi\in C_c^\infty(\R^d)$,
\begin{equation}\label{solfund}
\int_{\R^d} g_k(y) \Delta_k\varphi(y)w_k(y)dy=-\varphi(0).
\end{equation}
It is well known from \cite{mt1} that,
\begin{equation}\label{defgk}
g_k(y)=c_k\Gamma(\lambda)2^{\lambda-1}|y|^{-2\lambda}.
\end{equation}
\begin{theo}\label{hypoell}
Let $h\in C(D)$ and $f\in C^\infty(D)$. Assume that for every  $\varphi~\in~C_c^\infty(D)$,
$$\int_D h(x) \Delta_k\varphi(x)w_k(x)dx=\int_D f(x)\varphi(x)w_k(x)dx.$$
 Then $h\in C^\infty(D)$.
\end{theo}
\begin{proof} It suffices to prove that $h\in C^\infty(U)$, for every   $W$-~invariant open set $U$ such that $\overline U\subset D$.

\par{\it{Step 1.}} Assume first that $f=0$ on $D$. Choose $n_0\geq 1$ such that for every $x\in U$, $\overline B(x,\frac 1{n_0})\subset D$. For every $n\geq n_0$,
let   $\phi_n$ be  as in  (\ref{phin}). Then,   the function $h_n$ defined on $U$ by
$$h_n(x):=\int_{D} h(y)\tau_{-x}\phi_n(y)w_k(y)dy,$$
is infinitely differentiable on $U$ and by (\ref{txi}) for every $x\in U$,
$$\Delta_kh_n=\int_{D}h(y)\Delta_k(\tau_{-x}\varphi_n)(y)w_k(y)dy.$$ On the other hand, it follows from (\ref{eq}) that for every $x\in U$,
\begin{eqnarray*}
h_n(x)&=&d_k\int_0^{\frac 1n}{\phi_n}(t)t^{2\lambda+1}M_{x,t}(h)dt\\
&=&cd_k\int_0^1\phi(1-u^2)u^{2\lambda+1}M_{x,\frac un}(h)du.
\end{eqnarray*}
This yields that $(h_n)_{n\geq n_0}$  is uniformly bounded on $U$   and  converges pointwise  to $h$ on $U$.
Hence, in view of  Lemma~\ref{theconver}, $h\in C^\infty(U)$ and $\Delta_kh=0$ on $U$.
\par{\it{Step 2.}} We now turn to the general case where $f$ is not trivial.
Let $v~\in~C^\infty_c(\R^d)$ such that $v=f$ on $U$ and define $\psi$ on $\R^d$ by
$$\psi(x):=\int_{\R^d} g_k(y)\tau_xv(y)w_k(y)dy,$$
where $g_k$ is given by (\ref{defgk}).
  Using  spherical coordinates, it easily seen that  the function $g_k w_k$ is locally Lebesgue  integrable on $\R^d$. Thus, $\psi~\in~C^\infty(\R^d)$. Furthermore, it follows from (\ref{txi}) and (\ref{solfund}) that
$\Delta_k\psi= -f$ on $U$.
 Then, for every $\varphi\in C^\infty_c(U)$,
$$
  \int_{\R^d} (h(x)+\psi(x))\Delta_k\varphi(x)w_k(x)dx = \int_{\R^d} (f(x)+\Delta_k\psi(x))\varphi(x)w_k(x)dx = 0.
$$
 Whence, the first step  yields that $h+\psi$ is infinitely differentiable on $U$ which finishes  the proof.\end{proof}
 \vskip 5mm

We note that  the previous theorem was already proved by H.~Mejjaolli and K.~Trimèche \cite{mt2} using Sobolev spaces  associated with the Dunkl operators.\\

{\it{Proof of Theorem \ref{promoy} }}
  Statement $\it{(a)}$ follows from $\it{(b)}$ by means of Lemma~\ref{moycinfty} and Proposition~\ref{propmoycoro}. Assume now that $\it{(a)}$ holds.
Then, by \cite{dunkl3}, for every $\varphi~\in~C_c^\infty(D)$,
$$\int_D f(x) \Delta_k\varphi(x)w_k(x)dx=\int_{D} \Delta_k f(x)\varphi(x)w_k(x)dx=0.$$

Use now  Theorem~\ref{hypoell} and Proposition~\ref{propmoycoro} to finish the proof.\hfill $\Box$
\vskip 5mm

\par In the following we shall give a counterexample proving that Theorem~\ref{promoy} does not hold true if the  open set $D$ is  not $W$-invariant. To that end, let $d=1$ and  consider the root system $R=~\{\pm \sqrt2\}$.
 Then, the corresponding reflection group is given by $W=\{\pm id_{\R}\}$. Therfore, an open set    $U\subset  \R$ is $W$-invariant if and only if it is symmetric.
 \begin{prop}
 For every non symmetric open set  $U\subset\R$
 there  exists a  function $h:\R\rightarrow \R$ which is $\Delta_k$-harmonic   on $U$ but  does not satisfy
the mean value property on $U$.
\end{prop}
\begin{proof}
To abbreviate the notation we write $I_{x,r}:=]x-r,x+r[$. Let $x\in U$ and $r>0$ such that $\overline I_{x,r}\subset U$  and $\overline I_{-x,r}\cap \overline{U}
=\emptyset$.
Choose $f\in C_c^\infty(\mathbb R^d)$ such that
$f=-1$ on $\overline I_{-x,r}$ and  $f=0$ on $U$.  Since $g_kw_k$ is locally Lebesgue  integrable, we deduce that  the function $h$ defined
on $\R$  by
$$h(z)=\int_{\mathbb R}g_k(y)\tau_{z}f(y) w_k(y) dy,$$
is infinitely differentiable on $\R$.  Moreover, by (\ref{solfund}), for every $z\in\R$, $$\Delta_kh(z)=-\tau_{z}f(0)=-f(z).$$ Hence, $\Delta_kh=0$ on $U$.
On the other hand, for every $t\in ]0,r[$,
$$M_{x,t}(\Delta_kh)=-\int_\R f(y)d\sigma^k_{x,t}(y)=\sigma_{x,t}^k(\overline I_{-x,t}).$$
Moreover, it follows from \cite[Remarks 4.2]{rosler3} that
 $$\mbox{ supp\, }\sigma_{x,t}^k= \overline I_{x,t}\cup \overline I_{-x,t}.$$
Thus $\sigma_{x,t}^k(\overline I_{-x,t})>0$ and consequently $M_{x,t}(\Delta_kh)>0$. Whence, by (\ref{ddtmd})  the function $t\mapsto\frac {d}{dt}M_{x,t}(h)$ is
positive on $]0,r[$. Hence, $M_{x,t}(h)\neq h(x)$ for every $t\in]0,r[$, which means that  $h$ does not satisfy the mean value property on~$U$.\end{proof}
\vskip 5mm


\bibliographystyle{elsarticle-num}



\end{document}